\documentclass[a4paper,oneside]{amsart} 
\usepackage{amssymb}
\usepackage[latin1]{inputenc}
\usepackage{url,xcolor}
\usepackage{enumitem}
\setlist{leftmargin=25pt}

  \usepackage[left=2cm,right=2cm, top=3cm, bottom=3cm]{geometry}
\usepackage[all]{xy}

\usepackage[
        draft=false,
        colorlinks
]{hyperref}
\hypersetup{citecolor=blue!95!white,linktocpage}

\usepackage{listings} 
 
 \lstdefinelanguage{Magma}%
  {%
   otherkeywords={:=,+:=,-:=,*:=},%
   procnamekeys={function,func,intrinsic,procedure,proc,return},%
   morekeywords={true,false},%
   morekeywords=[2]{adj,and,cat,cmpeq,cmpne,diff,div,eq,ge,gt,in,is,join,le,lt,%
          meet,mod,ne,notadj,notin,notsubset,or,sdiff,subset,xor},%
   morekeywords=[3]{assigned,break,by,case,catch,continue,declare,default,%
          delete,do,elif,else,end,eval,exists,exit,for,forall,fprintf,if,local,%
          not,print,printf,quit,random,read,readi,repeat,restore,save,select,%
          then,time,to,try,until,vprint,vprintf,vtime,when,where,while},%
   morekeywords=[4]{clear,forward,freeze,iload,import,load},%
   morekeywords=[5]{assert,assert2,assert3,error,require,requirege,requirerange},%
   morekeywords=[6]{car,comp,cop,elt,ext,frac,hom,ideal,iso,lideal,loc,map,%
          ncl,pmap,quo,rec,recformat,rep,rideal,sub},%
      sensitive,%
      morecomment=[l]//,%
      morecomment=[s]{/*}{*/},%
      morestring=[b]"%
  }[keywords,procnames,comments,strings]%

\lstnewenvironment{code_magma}[1][]
{\lstset{basicstyle=\scriptsize\ttfamily, columns=fullflexible,
language=Magma,
keywordstyle=\color{red}\bfseries,
commentstyle=\color{blue},tabsize=4,
numbers=left, numberstyle=\tiny,
stepnumber=2, numbersep=5pt, frame=single, #1}}{}

\newcommand{\Aut}{\mathrm{Aut}}
\newcommand{\Stab}{\mathrm{Stab}}
\newcommand{\ord}{\mathrm{ord}}
\newcommand{\Fix}{\mathrm{Fix}}
\newcommand{\Cl}{\mathrm{Cl}}
\newcommand{\im}{\mathrm{Im}}

\newtheorem*{thm}{Theorem}

\newtheorem{theorem}{Theorem}[section]
\newtheorem{lemma}[theorem]{Lemma}
\newtheorem{proposition}[theorem]{Proposition}

\theoremstyle{definition}     
\newtheorem{definition}[theorem]{Definition}

\newtheorem*{claim}{Claim}

\theoremstyle{remark}
\newtheorem{remark}[theorem]{Remark}
\numberwithin{equation}{section}

\title[On the minimal model of semi-isogenous  mixed surfaces]{On the minimal model of semi-isogenous  mixed surfaces}

\makeatletter

\@addtoreset{equation}{section}
\makeatother

\author[D. Frapporti]{Davide Frapporti}
\address{
Politecnico di Milano,
Dipartimento di Matematica,
Piazza Leonardo da Vinci 32,
I-20133 Milano, Italy
}
\email{davide.frapporti@polimi.it}


\thanks{The author thanks  C.~Glei\ss ner and R.~Pignatelli  for fruitful conversations, and N. Cancian and 
 S. Coughlan for inspiring discussion and their  careful reading of the paper.	The author is a  member of GNSAGA of INdAM and 
 acknowledges the support of 
``INdAM - GNSAGA Project'' - CUP E53C25002010001. }

\date{\today}

\begin{document}

\begin{abstract} 
The aim of this note is to determine the  minimal models of the semi-isogenous mixed surfaces
with $\chi(\mathcal O)=1$ and $K^2>0$ constructed in \cite{CF18}.
In order to do this, we further develop  the idea of \textit{orbit divisors}, introduced in \cite{FL19} 
 to construct effective divisors on surfaces isogenous to a product of mixed type,
extending it to  semi-isogenous mixed surfaces.
\end{abstract}
\maketitle

\section*{Introduction}
\label{intro}

When we construct a ``new'' surface of general type, one of the first natural questions is: ``Is it minimal?''. 
In other words,  ``does it contain exceptional curves of the first kind? If yes, what is its minimal model?''.
These questions are natural, being the first step to addressing classification and geographical problems.
Although we know that a minimal model exists, to determine it, also to identify the
exceptional curves of the first kind (i.e., the $(-1)$-curves)
on our model, is in general not an easy task.

Recall that a  minimal surface of general type $S$ satisfies $K^2_S\geq 1$, $\chi(\mathcal O_S)\geq 1$,
the 
Bogomolov-Miyaoka-Yau inequality $K^2_S\leq 9\chi(\mathcal O_S)$, and if $q(S)>0$ the Debarre inequality $K^2_S\geq 2p_g(S)$.
In the boundary case $\chi(\mathcal O_S)=1$ (i.e., $p_g(S)=q(S)$)  the  above inequalities 
 imply  $p_g(S) = q(S)\leq 4$ and $2p_g(S) \leq K^2 \leq9$.
By the main theorem of \cite{Beau82}, if $S$ is a minimal surface of general type
with $p_g(S)=q(S)=4$, then $S$ is the product of two genus 2 curves.
The case $p_g(S)=q(S)=3$ is classified too (see \cite{CCML98,Pir02,HP02}):
$S$ is either  the symmetric square of a genus three curve $S=C_3^{(2)}$, or 
$S=(C_2 \times C_3)/\langle \nu\rangle $ where $C_g$ is a curve of genus $g$ and $\nu$  is an involution  
acting on $C_2$ as an elliptic involution, freely on $C_3$, and diagonally on the product.
The classification  becomes  more complicated as the value of $p_g$ decreases, 
and the case $p_g = q \leq 2$ is still rather unknown.

In recent years there has been growing interest in those surfaces having low invariants (e.g.~$\chi(\mathcal O)=1$),  and  several new surfaces of general type with $p_g=q$ have been constructing by considering quotients of  products of two curves by the action of a  finite group: see
\cite{BC04,BCG08,BCGP12,BP12,BP15,Frap13,FP15} for $p_g=0$, 
\cite{CP09,Pol07,Pol09,MP10,FP15} for $p_g=1$,
\cite{pen11,zuc03} for $p_g=2$. 

In  these articles the authors work under the assumption that the group action is free outside of a finite set of points, i.e., the induced map onto the quotient is \textit{quasi-\'etale}.

In \cite{CF18}, we  moved to the non-quasi-\'etale case. More precisely,
 the following situation 
was considered:
let $C$ be a smooth projective curve of genus $g(C)$ and $G$ a finite subgroup
of $\Aut(C)^2\rtimes \mathbb Z_2$  whose action is \textit{mixed}, i.e., there are elements in $G$ exchanging the two
natural isotrivial fibrations of $C\times C$.   Let $G^0\triangleleft G$ be the index two subgroup $G\cap\Aut(C)^2$, 
i.e., the subgroup consisting of those elements that do not exchange the factors.
Assuming that $G^0$ acts freely on $C\times C$, then $X:=(C\times C)/G$ is smooth  and it is a \textit{semi-isogenous mixed surface} (see Section \ref{sec:SIMS}).
The basic example  is $C^{(2)}$, the symmetric square  of a curve $C$, e.g.~${\mathbb P^1}^{(2)} \cong\mathbb P^2$.

Dropping any assumption on the $G$-action on $C\times C$ the \textit{mixed quotient}
$X:=(C\times C)/G$ may be singular and the minimal resolution $S\to X$ of its singularities
is  called \textit{mixed surface}.

In the above mentioned papers, the minimality of the new surfaces, respectively the construction of their minimal model, has been proved  using ad hoc arguments. 
In \cite[Theorem 4.5]{FP15} we proved that any irregular mixed quasi-\'etale surface of general type is minimal,
and by \cite[Theorem 3]{Pigna20} any mixed surface with irregularity $q\geq 3$ is minimal.

In \cite{CF18} we   classified semi-isogenous mixed surfaces  having 
$K^2>0$ and $p_g= q \leq 2$. 
Besides $\mathbb P^2$ there are 15 other  families with $p_g=q=0$, $K^2\in\{2,6,8\}$;
35 families with $p_g=q=1$, $K^2\in\{2,4,6,7,8\}$; 9 families with $p_g=q=2$,  $K^2\in\{2,4,6,7,8\}$; and 
all surfaces in these 59 families are of general type (see Theorem \ref{thm:mainCF18}).
We then proved that if $X$ is a semi-isogenous mixed surface  with $\chi(\mathcal O_X)=1$
and $K^2_X\geq 6$, then $X$ is minimal (cf.~Proposition \ref{bound_curves}).

In this paper we determine the minimal model in the remaining cases, namely  for all semi-isogenous mixed surfaces  with $p_g=q$ and $K^2\in \{2,4\}$.
 Combining our main theorem (=Theorem \ref{mainThm}) with the previous results, we  obtain the following.
\begin{thm}
Let $X:=(C\times C)/G$ be a semi-isogenous mixed surface of general type with  $p_g(X)=q(X)$, $K^2_X>0$ and 
$|G^0| \leq 2000, \neq 1024$, and let $X_{min}$ be its minimal model. Then the followings hold.
\begin{itemize}
\item If $K^2_X\geq 6 $, then  $X$ is  minimal.
\item If  $K^2_X=4$, then $X$ is not minimal and $K^2_{X_{min}}=5$.
\item If  $K^2_X=2$, then $X$ is not minimal and $K^2_{X_{min}}=4$.
\end{itemize} \end{thm}

To search for exceptional curves of the first kind on semi-isogenous mixed surfaces, we improve and  generalize the
idea  of  ``orbit divisors'' (see Section \ref{sec:OD}) introduced in \cite{FL19},
to construct effective divisors on surfaces isogenous to a product of mixed type (i.e.,  $G$ acts freely on $C\times C$).
We remark that,  under certain assumptions, all
rational curves on a mixed quotient  arise in this way (see Proposition \ref{prop:structure}),  and we raise the question whether this holds in general.

 Let $X:=(C \times C)/G$ be a semi-isogenous mixed surface. Studying divisors on $X$ is equivalent to studying
  $G$-invariant divisors on $C \times C$. 
  Therefore, we want to construct several $G$-invariant divisors on $C \times C$ in 
  order to increase our chances of finding  $(-1)$-curves.
 Let $H$ be a subgroup of $\Aut(C)$ containing $G^0$  and  for each $h \in H$, 
 let $\Gamma_h:= \{ (x,h(x)) \mid x \in C \}$  denote its graph. The $G$-action on $C\times C$ induces an action on the set  $\{ \Gamma_h \mid h \in H \}$, and from  
 the $G$-orbits of this set, we  obtain several $G$-invariant divisors on $C \times C$, which descend to effective divisors on $X$.  We call these divisors \textit{orbit divisors}.
Since we can compute the intersection numbers of these $G$-invariant divisors on $C \times C$,  
we can compute self-intersection and intersection with $K_X$ of the orbit divisors, and possibly identify $(-1)$-curves.

A natural candidate for $H$ is the group $G^0$ itself, but this does not always produce enough curves on $X$.
For example, if the curve $C$ is a covering of an elliptic curve $E$ branched over 2 points  of multiplicity $2$,
or an \'etale covering of a curve $C'$ of genus 2,   we need to consider a larger group of  automorphisms.
We show indeed  that   the involution of $E$ swapping the 2 branch points (resp. the hyperelliptic involution of $C'$)
lifts to an automorphism of $C$ compatible with the $G^0$-action, whence  $\Aut(C)$ contains the  subgroup $H$ 
generated by $G^0$ and a lift of the involution (see Section \ref{sec:lifting}).

The main theorem is then proved  case by case, constructing many
 orbit divisors and checking their intersection numbers  (Section \ref{sec:proof_main}).
 Some of the computations are performed using MAGMA \cite{MAGMA}: the scripts are available on the author web-page:
\url{https://davidefrapporti.github.io/Publications.html}.

\medskip

 {\bf Notation.} We work over $\mathbb{C}$ and use the standard notation in surface theory:
for a smooth projective  surface $X$ we denote by $p_g(X):=h^2(X, \mathcal O_X)$ its
geometric genus, by $q(X):= h^1(X, \mathcal O_X)$ its irregularity, 
by $\chi(\mathcal{O}_X)=1-q(X) +p_g(X)$ its holomorphic Euler-Poincar\'e characteristic, 
and by $K^2_X$ the self-intersection of its  canonical divisor.
For a smooth compact curve (Riemann surface)  $C$ we denote by $g(C)$  its genus.

We shall also borrow a standard notation from group theory: we denote by $\mathbb  Z_n$ the cyclic group of order $n$, by   $S_n$ the symmetric group on $n$ letters,   by $Q$ the group of quaternions, 
by $D_n$ the 	dihedral group of order $2n$,
 and by	 $D_{p,q,r}$ the finite group $\langle x,y\mid x^p= y^q=1, xyx^{-1}=y^r\rangle$.

\section{Semi-isogenous Mixed Surfaces}\label{sec:SIMS}

In this preliminary section we recall the definition of  a semi-isogenous mixed surface and a few general results on them; we refer to \cite{CF18,Pigna20} for further details.

Let $C$ be  a smooth projective curve of genus $g(C)$, let 
$\sigma$ be the involution on $C\times C$ swapping the two factors:
 $\sigma(x,y)=(y,x)$ for $x,y\in C$, and let $\Aut(C)^2< \Aut(C\times C)$ be the subgroup of automorphisms of $C\times C$  preserving the two factors.
 We denote by $\Aut(C)(2)$ 
 the smallest subgroup of   $ \Aut(C\times C)$   containing $\Aut(C)^2$ and $\sigma$, i.e., $\Aut(C)(2):= \Aut(C)^2\rtimes_\psi \mathbb Z_2\,,   \psi: 1\mapsto\sigma$.
Note that  if  $g(C)\geq 2$ then $\Aut(C)(2)=\Aut(C\times C)$ (see \cite[Corollary 3.9]{Cat00}) and 
it  is a finite group.

\begin{definition}
Let $G$ be a finite subgroup of $\Aut(C)(2)$. 
We say that its action on $C\times C$ is 
\textit{mixed} if $G$ is not contained in $\Aut(C)^2$, i.e., there are elements in $G$ 
exchanging the two isotrivial fibrations of $C\times C$.  

We denote by  $G^0\triangleleft_2 G$ the index two subgroup $G\cap\Aut(C)^2$, 
i.e., the subgroup consisting of those elements that preserve the factors.
The action is \textit{minimal } if the group $G^0$ acts faithfully on both factors.
\end{definition}

\begin{remark} 	
  By \cite[Remark 3.10]{Cat00} and \cite[Remark 3.3]{Frap13} we may assume (and  from now on we  do)  that 
the action is minimal. In this case we identify  $G^0<\Aut(C) \times \Aut(C)$ with its projection onto the first factor.
\end{remark}

We have the  following normal form of minimal mixed actions:

\begin{theorem}[cf.~{\cite[Proposition 3.16]{Cat00}}]
\label{teoazione_sm} 
Let $G$ be a finite subgroup of $\Aut(C)(2)$ whose action is minimal and mixed. Fix $\tau'
\in G\setminus G^0$; let $\tau:=\tau'^2\in G^0$ and let $\varphi\in\Aut(G^0)$ be defined by $\varphi(h):=\tau'h\tau'^{-1}$.
Then, up to a coordinate change, $G$ acts as follows:
\begin{equation}
\label{azione_sm}
\left\{\begin{array}{r}
g(x,y)=(gx,\varphi(g)y)\\
\tau' g(x,y)=(\varphi(g)y,\tau gx)
\end{array}\right.
\qquad \textrm{for }g\in G^0\,.
\end{equation}

 Conversely, for every finite subgroup $G^0 <\Aut(C)$ and $G$ extension of degree $2$ of $G^0$, fixed $\tau'\in G\setminus G^0$ and defined
 $\tau$ and $\varphi$ as above, (\ref{azione_sm}) defines a minimal mixed action on $C\times C$.
\end{theorem}

\begin{definition}
Let $C$ be a smooth projective curve and let $G<\Aut(C)(2)$ be a finite group  whose action is mixed.
The quotient surface $X:=(C\times C)/G$ is a \textit{mixed quotient}, and its 
minimal resolution of the singularities $S\to X$ is a \textit{mixed surface}.

 If $G^0\triangleleft_2 G$ acts freely on $C\times C$, 
then the quotient surface $X:=(C\times C)/G$ is  called a \textit{semi-isogenous mixed surface}.
If the whole group $G$ acts freely on $C\times C$, then  $X$ is   a \textit{surface isogenous to a product} of  mixed type.
\end{definition}

\begin{proposition}[{\cite[Corollary 2.11]{CF18}}]
Every semi-isogenous mixed surface  is smooth.
\end{proposition}

Let $X:=(C\times C)/G$ be a semi-isogenous mixed surface.
Since $G^0$ acts freely, any non-trivial element of $G$ having fixed points on $C\times C$ lies in
 $G\setminus G^0$ and must have order 2. 
\begin{definition}\label{O2}
Let $
  O_2:=\{g\in G\setminus G^0: g^2=1\}$, and  for each $g\in O_2$ let  $R_{g}$ denote
 its fixed locus: $R_g:=\Fix(g)$, i.e., $R_g$ is  the graph
of the automorphism $\tau' g\in G^0$: $R_g:=\{(x,(\tau' g) \cdot x) : x\in C\}$.

We denote  by $\Cl(g)$ the conjugacy class of $g\in G$ and we define $   \Cl(O_2):=\{\Cl(g) : g\in O_2\}$. 
\end{definition}

\begin{proposition}[cf.~{\cite[Section 3]{CF18}}]\label{ram_branch}
Let $X:=(C\times C)/G$ be a semi-isogenous mixed surface, and let
 $\eta \colon C\times C\to X$ be the quotient map.
Then:
\begin{itemize}
\item[i)] the ramification locus $\mathcal R$ of  $\eta $ is the disjoint union
of  the curves $R_g$,  for $g\in O_2$;    the  stabilizer  of  $R_g$ is $\langle g \rangle \cong \mathbb Z_2$.
  
\item[ii)]  the branch locus $\mathcal B$ of $\eta$   is the disjoint union 
 $\mathcal B=B_{g_1}\sqcup\cdots \sqcup B_{g_t}$, where $t:=|\Cl(O_2)|$,
 $\{g_1,\ldots, g_t\}$ is a set of representative of the conjugacy classes in $\Cl(O_2)$ and  $B_{g_i}:=\eta(R_{g_i})$ is an irreducible smooth curve of genus
$
   g(B_{g_i})=\frac{2(g(C)-1)}{|Z({g_i})|}+1.
  $
\item[iii)]

 $X$ has invariants  \[q(X)= g(C/G^0)\,, \qquad
  \chi(\mathcal O_X)  =\frac{(g(C)-1)^2}{|G|}-\frac {p_a(\mathcal{B})-1}2\,, \quad \text{ and } \quad
      K^2_X  =8 \chi(\mathcal O_X)-(p_a(\mathcal{B})-1) \,.
 \]
\end{itemize}
\end{proposition}

\begin{remark}
The  formula for the irregularity holds for any mixed surface $S$ resolution of a mixed quotient $S\to (C\times C)/G$, namely 
$q(S)=g(C/G^0)$ (see {\cite[Proposition 1.7]{Pigna20}}).
\end{remark}
\subsection{The Classification of semi-isogenous mixed surface with $\chi(\mathcal O)=1$}\label{sub:class}

By Theorem \ref{teoazione_sm}, a  semi-isogenous mixed surface is equivalent to  the following data:
\begin{enumerate}
\item a curve $C$ together with a  faithful $G^0$-action;
\item a degree 2 extension $G$ of $G^0$, such that the induced $G^0$-action on $C\times C$ is free.
\end{enumerate}

The translation into algebraic terms  of semi-isogenous mixed surface is then  accomplished through  
the  theory   of Galois coverings between projective curves (cf.~\cite[Section III.3, III.4]{Mir}).

\begin{definition}
Given integers  $g'\geq 0$, $m_1, \ldots, m_r > 1$  and   a finite group $H$, a \textit{generating vector} for $H$ of type $[g';m_1,\ldots ,m_r]$ is a $(2g'+r)$-tuple of
elements of $H$:
\[V:=(d_1,e_1,\ldots, d_{g'},e_{g'};h_1, \ldots, h_r)\]
such that $V$ generates $H$, $\prod_{i=1}^{g'}[d_i,e_i]\cdot h_1\cdot h_2\cdots h_r=1$ and $\ord(h_i)=m_i$.
\end{definition}

\begin{theorem}[Riemann Existence Theorem]\label{RET}
A finite group $H$ acts as a group of automorphisms of
some compact Riemann surface $C$ of genus $g(C)$, 
if and only if  there exists a generating vector
$V$
for  $H$ of type $[g';m_1,\ldots,m_r]$,
such that  the  Hurwitz formula holds:
\[
2g(C) -2 = |H| \bigg( 2g'-2+\sum_{i=1}^r \frac{m_i - 1}{m_i} \bigg). 
\]
In  this case $g'$ is the genus of the quotient Riemann surface $C' := C/H$
and the $H$-covering $c:C\to C'$ is branched over $r$ points $\{x_1, \ldots , x_r\}$ with
branching indices $m_1, \ldots ,m_r$, respectively.
\end{theorem}

\begin{remark}\label{rem:Stab}
a)  The points of the  fiber $c^{-1}(x_i)$ are in 1-1 correspondence with the cosets $\{gK_i\}_{g\in H}$ of the cyclic group $K_i:=\langle h_i\rangle$
and the stabilizer of the point $y\in c^{-1}(x_i)$ corresponding to $gK_i$ is $gK_ig^{-1}$.

b) The \textit{stabilizer set} $\Sigma_V$ of $V$  is  the subset of $H$ consisting of the automorphisms of $C$ having  fixed points:
\[\Sigma_V:= \bigcup_{g\in H}\bigcup_{j\in \mathbb Z}\bigcup_{i=1}^r \{ g\cdot h_i^j\cdot g^{-1}\}\,.\]
\end{remark}

To give a  semi-isogenous mixed surface is then equivalent to give the following:
 \begin{enumerate}
 \item a finite group  $G^0$, 
a curve $C'$ (which will be $C/G^0$), $r$ points on $C'$ and 
a  generating vector $V$ for $G^0$  of type $[g(C');m_1,\ldots,m_r]$ (these give the Galois-covering $C\to C/G^0=C'$).

\item   a degree 2 extension $G$ of $G^0$ (this defines the $G$-action on $C\times C$), such that the following holds:
let $\varphi\in\Aut(G^0)$ as in Theorem \ref{teoazione_sm},
then the stabilizer sets $\Sigma_V$ and $\Sigma_{\varphi(V)}(=\varphi(\Sigma_V))$ are \textit{disjoint}, 
i.e., $\Sigma_V\cap\Sigma_{\varphi(V)}=\{1\}$ (this ensures that $G^0$ acts freely on  $C\times C$).
 \end{enumerate}
 
This algebraic description has been used in \cite{CF18} to  obtain the following classification
(see \cite{BCG08,Frap13,CP09,pen11} for the isogenous case: $K^2=8$).

\begin{theorem}[{\cite[Theorems A, B, C]{CF18}}]\label{thm:mainCF18}
\quad
\begin{itemize}
\item 
Let $X:=(C\times C)/G$ be a  semi-isogenous mixed surface  with  $p_g(X)=q(X)=0$ and $K^2_X>0$,
such that $|G^0|\leq 2000, \neq 1024 $.
Then either $X$ is $\mathbb P^2$, or $X$ is of general type.\\
The latter form 15 families: 5 with $K^2=8$, 8 with $K^2=6$ and 2 with $K^2=2$.

\item The semi-isogenous mixed surfaces  with  $p_g=q=1$ and $K^2>0$,
are of general type and form 35 families: 3 with $K^2=8$, 2 with $K^2=7$, 12 with $K^2=6$, 2 with $K^2=4$ and 16 with $K^2=2$.

\item 
The semi-isogenous mixed surfaces  with $p_g=q=2 $ and $K_X^2>0$,
are of general type and form 9 families: 1 with $K^2=8$, 1 with $K^2=7$, 3 with $K^2=6$, 1 with $K^2=4$,  and 3 with $K^2=2$.
 
\end{itemize}
These 59 families are collected in
Appendix \ref{app_tables}.

\end{theorem}

The condition $|G^0|\leq 2000, \neq 1024 $ in Theorem \ref{thm:mainCF18}  is a computational assumption,
since its proof relies on the MAGMA \cite{MAGMA} database of small groups, 
which has some  technical limitations (cf.~\cite[Remark 6.5]{CF18}).
\subsection{On the minimality}\label{On_min}
In \cite{CF18} we could only partially answer the question of minimality for the surfaces we constructed. 
More precisely, using the Algebraic Hodge Index Theorem we  proved:

\begin{proposition}[{\cite[Proposition 7.6]{CF18}}]
\label{bound_curves}
Let $X$ be a semi-isogenous mixed surface of general type with invariants $p_g(X)=q(X)$ 
and let $\rho\colon X\to X_{min}$ be  the projection to its minimal model.
Then
\begin{itemize}
\item for $K^2_X\in\{6,7,8\}$, $\rho$ is the identity map: $X=X_{min}$;
\item for $K^2_X\in\{4,5\}$, $\rho$ is the contraction of at most one $(-1)$-curve;
\item for $K^2_X\in\{2,3\}$, $\rho$ is the contraction of at most two $(-1)$-curves;
\item for $K^2_X=1$, $\rho$ is the contraction of at most three $(-1)$-curves.
\end{itemize}
\end{proposition}

In particular if $X$ is a semi-isogenous mixed surface of general type with $\chi(\mathcal O_X)=1$
and $K^2_X\geq 6$, then $X$ is minimal.
In most 
of  the remaining cases we used other classification results
to conclude that the surfaces are not minimal, but we did not determine their minimal 
model, except for other 
4 families, namely those  having $p_g(X)=q(X)=2$ and $K^2_X \in \{2,4\}$
(see {\cite[Section 7]{CF18}}).

The open cases are  marked by (\#) in 
Table \ref{q0} and \ref{q1} of Appendix \ref{app_tables}, and the  aim of this paper is to explicitly construct $(-1)$-curves on  these surfaces,
 in order to determine their minimal model.

\section{Orbit Divisors on  Semi-isogenous Mixed Surfaces}\label{sec:OD}

In this section we explain a method to construct effective divisors on a semi-isogenous mixed surface, or more generally 
on a mixed quotient, generalizing the idea introduced in cf.~\cite{FL19}.
%
As we will see, under certain assumption, all rational curves on a mixed quotient $X$  with irregularity $q=1$ or $2$ appear in this way.

Let $C$ be a smooth curve.
For  $f\in \Aut(C)$ we denote by $\Gamma_f $ its graph: \[\Gamma_f:=\{(x,f(x) )\in C \times C\mid x \in C\}\,.\]

Let $G$ be a finite subgroup of $\Aut(C)(2)$ whose action  on $C\times C$ is mixed and 
let $H$ be a subgroup of $\Aut(C)$ containing $G^0$: $G^0<H< \Aut(C)$.
Using  formula  (\ref{azione_sm}),  one can  show that the $G$-action on $C\times C$ induces a $G$-action on the set $\{\Gamma_f \mid f \in H\}$. More precisely, in the notation of Theorem  \ref{teoazione_sm}, the following holds.

\begin{lemma}[{\cite[Lemma 3.1]{FL19}}]
Let $h$ be in $G^0$, then $h(\Gamma_f)= \Gamma_{\varphi(h)fh^{-1}}$ and $ \tau'h(\Gamma_f)= \Gamma_{\tau h  f^{-1}\varphi(h^{-1})}$.
\end{lemma}

In this way, we partition the set $\{\Gamma_f \mid f \in H\}$ in $G$-orbits, each one giving  a $G$-invariant effective divisor on $C\times C$.
For $f\in H$, the sum (taken with the reduced structure) of the divisors in its  orbit
$\{\gamma(\Gamma_{f})\}_{\gamma \in G}$   gives the effective divisor 
\[\tilde D:= \left(\sum_{\gamma \in G}\gamma(\Gamma_{f})\right)_{red}=\Gamma_{f_1}+\ldots+\Gamma_{f_n}\,,\]
where  $n$ is the index of the subgroup $\{\gamma \mid \gamma(\Gamma_{f})=\Gamma_{f}\}\subset G$.

Let $\eta\colon C\times C\to X:=(C\times C )/G$ be the quotient map.
The divisor $\tilde D$  is $G$-invariant, so it yields an effective 
 divisor $D:=\eta_*(\tilde D)_{red}$ on  $X$. 
 Note that $D$ is irreducible, as the image of an irreducible divisor: $D= \eta(\Gamma_f)$.
Each $G$-orbit yields a distinct effective  irreducible divisor on $X$, this motivates the following name.  
\begin{definition} The effective divisor
 $D= \eta(\Gamma_f)$ is called the  \textit{orbit divisor} induced by  $f$.
\end{definition}

As mentioned before, the motivation behind this definition is that, under certain assumption, all rational curves  on a mixed quotient $X$  with irregularity $q=1$ or $2$   (cf.~\cite[Proposition 1.8, Theorem 4.3]{Pigna20} for the case $q\geq 3$)
arise as orbit divisors,
 as explained in the following proposition.

\begin{proposition}\label{prop:structure}
Let $S\to X:=(C\times C)/G$ be a mixed surface with irregularity $q:=q(S)$ and let $E$ be a rational curve on the  mixed quotient $X$.
Assume that one of the following holds:
\begin{itemize}
\item  $q=2$ and   the hyperelliptic  involution $\varsigma$ of the genus 2 curve $C/G^0$ lifts to an automorphism of $C$;
 
\item $q=1$ and  the involution $\varsigma: x \mapsto -x$ of the elliptic curve $C/G^0$ lifts to an automorphism of $C$.
\end{itemize}
Then $E$ is the orbit divisor induced by  a lift $\hat \varsigma \in \Aut(C)$ of $\varsigma$.

\end{proposition}

\begin{proof}
Let $C'$ be the genus $q$ curve $C/G^0$ and let $C'^{(2)}$ be its  symmetric square.

By construction (cf.~\cite[Remark 1.6]{Pigna20}), there is a finite dominating morphism $\pi:X\to C'^{(2)}$, which fits into the commutative diagram:
\begin{equation}
\xymatrix{
C\times C \ar[rr]^Q \ar[d]_\eta &&   C' \times C'\ar[d]_s\\
X \ar[rr]^\pi &&C'^{(2)} \ar[r]^a & J(C')
}
\end{equation}
where  $s \colon C'\times C' \to C'^{(2)}$ is the quotient map, $Q$ is the product map $Q:=c\times c$, being $c\colon C\to C'=C/G^0$, and 
$a\colon C'^{(2)}\to J(C')$ is  the  Abel-Jacobi map.
Since $\pi$ is a finite morphism,  $ \pi(E)$ is a rational curve on $C'^{(2)}$.

\smallskip 

If $g(C')=2$ we borrow an argument from \cite[Proposition 1.5]{PP16}.
In this case  
$a\colon C'^{(2)}\to J(C')$ is the contraction of the
$(-1)$-curve $E'$, image of  the graph $\Gamma:=\{(x,\varsigma(x))\mid x \in C'\}$ of the hyperelliptic involution $\varsigma$ of $C'$, via the quotient map $s \colon C'\times C' \to C'^{(2)}$.
Since $J(C')$ is a torus,  $ C'^{(2)}$ does not contain other rational curves, whence $ \pi(E)=E'$.

Since every irreducible component of $\eta^{-1}(E)$ dominates $\Gamma$ and the involution $\varsigma $ lifts to 
an automorphism $\varsigma \in \Aut(C)$, we have that every irreducible component of $\eta^{-1}(E)$
has the form $\{(u, g \hat\varsigma(u) )\mid u \in C\}= \Gamma_{g\hat \varsigma}$ for a $g \in G^0$ and the first claim follows.

\smallskip

If $g(C')=1$, then the rational curves on $C'^{(2)}$ are precisely  the fibers of   the Abel-Jacobi map ($a(\{u,v\})=u+v$),
whence $a(\pi( E))$ is a point  $p\in C'$ and, up to translation, we may assume $p$ to be the neutral element $0\in C'$.
It follows that $\pi(E)$ is the image of $\Gamma:=\{(x,-x)\mid x \in C'\}$, the graph of the involution $\varsigma\colon x\mapsto -x$ of $C'$, via the quotient map $s\colon C'\times C' \to C'^{(2)}$.
Arguing as above we get the second claim.
\end{proof}

We pose  the question whether a similar description for the rational curves holds in general.

\smallskip

We conclude this section computing the intersection numbers for  orbit divisors on semi-isogenous mixed surfaces.

\begin{lemma}\label{lem:intOD}
Let $X:=(C\times C )/G$ be a  semi-isogenous mixed surface, let $ \mathcal R:= \sum_{g\in O_2} \Gamma_{\tau' g}$ be the ramification divisor of $\eta: C\times C\to X$ and
let $H$ be a subgroup of $\Aut(C)$ containing $G^0$.

Let $D:=\eta_*(\Gamma_{f_1}+\ldots+\Gamma_{f_n})_{red}$ and 
$D':=\eta_*(\Gamma_{f'_1}+\ldots+\Gamma_{f'_m})_{red}$ be two distinct orbit divisors on $X$ induced  respectively by $f, f' \in H$.
Then 
\[ D.D' = \frac{\alpha}{|G|} \sum_{i=1}^{n}  \sum_{j=1}^{m} \Gamma_{f_i} . \Gamma_{f'_j}\,, \]
where $\alpha:=0$ if $D$ and $D'$ are both branch curves,  $\alpha:=2$ if only one of them is a branch curve,
and $\alpha:=1$ if neither of them is a branch curve.

If $D$  is not a branch curve of $\eta$, then
\[
D^2 = \frac{-2(g(C)-1)n}{|G|}  +\frac{2}{|G|}\sum_{1\leq i<j\leq n} \Gamma_{f_i}. \Gamma_{f_j}  \,,\quad \textrm{ and }\quad
K_X.D=\frac{ 4(g(C)-1)n }{|G|}- \frac{1}{|G|}\sum_{j=1}^n\sum_{g\in O_2}   \Gamma_{\tau' g}. \Gamma_{f_j}\,.
\]
If $D$  is  a branch curve of $\eta$, then
\[
D^2 = \frac{-8(g(C)-1)n}{|G|}  +\frac{8}{|G|}\sum_{1\leq i<j\leq n} \Gamma_{f_i}. \Gamma_{f_j}  \,, \quad \textrm{ and }\quad 
K_X.D =\frac{ 12(g(C)-1)n }{|G|}\,.
\]

\end{lemma}

\begin{proof} If $D$ and $D'$ are branch curves, then $D.D'=0$ since they are disjoint  by Proposition \ref{ram_branch}. 
Otherwise let $\tilde D:=\Gamma_{f_1}+\ldots+\Gamma_{f_n}$ and set $m:=2$ if $D$ is a branch curve for $\eta$, and   $m:=1$ if not.
 By definition we have $\eta^* D= m\tilde D$ and the same goes for $D'$.
By  the projection formula the first formula follows and  we obtain 
\[
D^2 =\frac{1}{|G|}(\eta^*D)^2= \frac{1}{|G|} \left(m \cdot \sum_{i=1}^n \Gamma_{f_i}\right)^2=
 \frac{m^2}{|G|} \left(\sum_{i=1}^n \Gamma_{f_i}^2 +2\sum_{1\leq i<j\leq n} \Gamma_{f_i}. \Gamma_{f_j} \right)\,.\]

The canonical divisor $K_{C\times C}$ of $C\times C$ is numerically equivalent to 
$2(g(C)-1)(F_1+F_2)$, where $F_1,F_2$ are general fibers of the projection of $C\times C$ respectively onto  the first and  the second coordinate, whence the values $\Gamma_{f_i}^2$ are easily computed using the adjunction formula $K_{C\times C}. \Gamma_{f_i}+ \Gamma_{f_i}^2 =2(g(C)-1)$:
\[K_{C\times C}. \Gamma_{f_i}= 2(g(C)-1)(F_1+F_2). \Gamma_{f_i}= 4(g(C)-1)\,,\quad 
\textrm{and so} \quad \Gamma_{f_i}^2 =  -2(g(C)-1)\,. \]
Both formulas for $D^2$ follow.

\smallskip

Since $\eta$ is branched along  some curves, each with ramification index 2, the ramification formula  implies $\eta^* K_X =K_{C\times C}-\mathcal R$ and the projection formula  yields
\[
  K_X.D=\frac{1}{|G|} \eta^*K_X. \eta^* D= \frac{m}{|G|} K_{C\times C}.\tilde D -\frac{m}{|G|} \mathcal R.\tilde D\,.
\]

The first summand  on the right hand side can be easily computed as above: 
$K_{C\times C}.\tilde D = 4(g(C)-1)n$, so we are left with  $\mathcal R.\tilde D$.
We distinguish 2 cases.

If $D$  is not a branch curve, then none of the $\Gamma_{f_i}$ is a component of $ \mathcal R$, whence
\[
\mathcal R.\tilde D = \sum_{i=1}^n \mathcal  R. \Gamma_{f_i}= \sum_{i=1}^n \sum_{g\in O_2} \Gamma_{\tau' g}.  \Gamma_{f_i}\,.
\]

 If $D$ is a branch curve, then the curves $\Gamma_{f_i}$ are ramification curves: irreducible components of $ \mathcal R$,
  and so
   $\tilde R:=  \mathcal R-\tilde D$  is an effective divisor.
  By  Lemma \ref{ram_branch} the ramification curves are pairwise disjoint,  whence
 \[\mathcal R. \tilde D= ( \tilde D + \tilde R). \tilde D=  \tilde D^2= \sum_{i=1}^n \Gamma_{f_i}^2=
 -2(g(C)-1)n \,.\]
\end{proof}

\begin{remark}
By {\cite[Lemma 3.3]{FL19}}, if  $f_1,  f_2\in H <\Aut(C) $ are different elements, 
then their graphs intersect transversally. Therefore, to determine $\Gamma_{f_1}.\Gamma_{f_2}= |\Gamma_{f_1}\cap\Gamma_{f_2}|=|\{x\in C \mid f_1(x)=f_2(x)\}|$, it is enough
to count the number of points on $C$ fixed by  the non-trivial automorphism $f:=f_1^{-1} f_2 \in H$.

This can be done  by applying the  Riemann Existence Theorem (Theorem \ref{RET})  to the covering   $C\to C/H$. 
According to Remark \ref{rem:Stab}  (we use here the same notation)
 the  automorphism  $f\in H$ fixes the point $gK_j (=:y)$ if and only if $ f\in  gK_jg^{-1} (= \Stab(y))$. Since every coset has $|K_j|=m_j$ representatives, we obtain
\[|\Fix( f)|=\sum_{j=1}^r \frac{1}{m_j}\sum_{g\in H} \textbf 1_{gK_jg^{-1}}( f)\,,\]
where
$\textbf 1_{gK_jg^{-1}}$ denotes  the indicator function of the set $gK_jg^{-1}$.

\end{remark}

\section{Lifting Automorphisms}\label{sec:lifting}

Motivated by Propositon \ref{prop:structure}, we provide  in this section sufficient conditions for  an involution on $C/G^0$ to  lift  to an automorphism of $C$.

\begin{lemma}\label{lemma:lifting} 
Let $C$ be a smooth projective curve, and let $H^0<\Aut(C)$ be a finite group
acting faithfully on $C$, such that a generating vector $V:=(a_1,b_1; c_1,c_2)$ for $H^0$ associated to  the Galois covering
$f_1\colon C\to C/H^0=:E$ is of type $[1; m,m]$, i.e., $E$ is an elliptic curve and $f_1$ is branched in 2 points
$p_1, p_2$ of index $m$. Then: 
\begin{itemize}
\item[i)] the involution $\iota\colon E\to E, x\mapsto (p_1+p_2)-x$ lifts to an automorphism $\hat \iota$ of $C$ compatible 
with $f_1$, i.e., $f_1\circ\hat \iota= \iota \circ f_1$.
\item[ii)] the group $H<\Aut(C)$ generated by $H^0$ and $\hat \iota $ is a degree 2 extension of $H^0$: $H\cong H^0\rtimes \mathbb Z_2$.
\item[iii)] The Galois covering $f\colon C\to C/H\cong \mathbb P^1$ is branched in 5 points:
the image (via the quotient map $f_2\colon E\to E/\langle \iota\rangle \cong \mathbb P^1$) of $p_1$ with multiplicity $m$  and of the four halves of $(p_1+p_2)$ with multiplicity $2$; i.e., the covering has a  generating vector of type $[0;2^4,m]$.
Moreover, such a generating vector has the form  $(-;h_1,h_2,h_3, h_4, c_1)$ where $h_j\in H\setminus H^0$.
\end{itemize}
\end{lemma}

\begin{proof}
The involution $\iota$ has four  fixed points: $\frac{p_1+p_2}2 +E[2]$ the four halves of $p_1+p_2 \in E$, hence 
the Galois covering $f_2\colon E\to  E/\langle\iota \rangle\cong \mathbb P^1$ is branched in four points $q_1,q_2,q_3,q_4$, each with ramification index $2$.
These four points together with $p:=f_2(p_1)=f_2(p_2)$ form the branch locus $B\subset \mathbb P^1$ of the  covering
$f:=f_2\circ f_1 \colon C\to \mathbb P^1$. 

\begin{claim} The map $f\colon C\to \mathbb P^1$ is a Galois covering, with Galois group $H$:
a degree 2 extension of $H^0$.\end{claim}

Taking the Claim for granted, the statement follows immediately.

i)-ii) Being a degree 2 extension of $H^0$, the group $H$ is generated by the subgroup $H^0$ and a single element
of $H\setminus H^0$, which is then  a lift of $\iota$ compatible with $f_1$. Moreover, the  short exact sequence
\[1\longrightarrow H^0 \longrightarrow H\longrightarrow \mathbb Z_2 \longrightarrow 1 \]
splits, because $h_i\in H\setminus H^0$ has order 2, whence $H\cong H^0\rtimes \mathbb Z_2$.

 iii) We have $\deg(f)=2\cdot |H^0|= |H|$,  and  each $q_i$ has $|H^0|$ preimages, i.e., ramification index $\frac{|H|}{|H^0|}=2$,
 while  $p$ has $2\cdot \frac{|H^0|}m$ preimages, i.e., ramification index $m$.
The stabilizers of the points in  $f^{-1}(p)$ are  conjugate to each other in $H$, 
and by assumption one of them is generated by $c_1$.
 On the other hand, a fixed    point in $f^{-1}(q_i)$  has  stabilizer of order 2 generated  by   an element in $ H\setminus H^0$,   since $H^0$  fixes no points outside $f_1^{-1}(p_1)$ and  $f_1^{-1}(p_2)$.
Therefore, a generating vector associated to the covering $f$ has the form  $(-;h_1,h_2,h_3, h_4, c_1)$, $ h_i \in H\setminus H^0$.
\end{proof}

\noindent\textit{Proof of Claim.} Let $x'\in \mathbb P^1\setminus B$ and  consider the monodromy map 
\[ \mu:\pi_1(\mathbb P^1\setminus B,x') \rightarrow \mathfrak S_F  \]
of the covering $f:C \to \mathbb P^1$, where $F:=f^{-1}(x')$.
Recall that for  a loop $\gamma$ in $\mathbb P^1\setminus B$ with base point $x'$, 
the bijection $\mu(\gamma)\in \mathfrak S_F$ of the fiber $F$ is defined as follows:
given  a point $y\in F$, we map it to the end point $ \hat \gamma_y(1)$ of the unique lift $\hat \gamma_y$ of $\gamma$ with  starting point $y$.

This map depends on the base point, and to change it, amounts to a relabelling of the element of $F$,
hence the image of $\mu$ as abstract group is well defined (see \cite[Section III.4]{Mir}).
The covering $f$ is then Galois
if and only if  $|\im(\mu)|=|F|=2|H^0|$, in this case $\im(\mu)$ is the Galois group of the covering.

Let $x_1$, $x_2$ be the 2 preimages of the base point $x'$ via the map $f_2$,
and fix a path $\tilde \beta:[0,1]\to E\setminus f_2^{-1}(B)$ with starting point $\tilde\beta(0)=x_1$ 
and end point $\tilde\beta(1)=x_2$.

This path descends to a loop $\beta:={f_2}_*(\tilde \beta)$ in $\mathbb P^1\setminus B$,
 and we set $b:=\mu(\beta)\in \mathfrak S_F$. 
  This permutation $b$ defines  a bijection between the following subsets of $F$: $A_1:=f_1^{-1}(x_1)$ and $A_2:=f_1^{-1}(x_2)$,
 and  it induces an isomorphism of the images of the monodromy maps
\begin{eqnarray*}
 \nu_1: \pi_1(E\setminus\{p_1,p_2\},x_1 ) &\to & H^0 \leq \mathfrak{S}_{A_1}\,,\\
\nu_2: \pi_1(E\setminus\{p_1,p_2\},x_2 ) &\to & H^0 \leq \mathfrak{S}_{A_2}\,.
 \end{eqnarray*}

Each loop in $\pi_1(\mathbb P^1\setminus B,x')$  lifts on $E$ 
either to loops  with starting  point $x_i$ ($i=1,2$), 
or  to paths joining $x_1$ and $x_2$.

Let $\gamma \in \pi_1(\mathbb P^1\setminus B,x')$ be a loop of the first type and let $\gamma_i$ be its lift 
on $E$ with base point $x_i$.
By definition of the monodromy maps we have that 
\[\mu(\gamma)=(\nu_1(\gamma_1), \nu_2(\gamma_2))= (g_\gamma,g'_\gamma) \in H^0 \times H^0 \leq  \mathfrak S_{A_1}\times  \mathfrak S_{A_2}\leq \mathfrak S_{F}\]
hence we can identify the pair $(g_\gamma,g'_\gamma) $ with its first entry  and 
the loops of the first type generate a subgroup $L\cong H^0$ inside 
$\mathfrak S_{A_1}\times  \mathfrak S_{A_2}\leq  \mathfrak S_F$.

Note that $b$ swaps $A_1$ and $A_2$, so  $b=\mu(\beta)\notin L$.
On the other hand the unique lift of  $\beta^2$ with starting point 
$x_1$ (or $x_2$) is a loop, whence $b^2=\mu(\beta^2)\in L$.

Let now $\delta \in \pi_1(\mathbb P^1\setminus B,x')$ be a loop, whose unique lift on $E$ with starting 
  point $x_i$ is  not a loop.
Then  $ \beta^{-1}*\delta$ lifts to loops with base points $x_i$ and so  \[\mu(\delta)=\mu(\beta *\beta^{-1}*\delta)=\mu(\beta)\cdot \mu(\beta^{-1}*\delta)=
b\cdot(g_{\beta^{-1}*\delta}, g'_{\beta^{-1}*\delta})\in b\cdot L\,.\]

Finally, we have $b \cdot (g_\gamma,g'_\gamma)\cdot b^{-1}=
\mu(\beta*\gamma*\beta^{-1})= \mu (\gamma')= (g_{\gamma'},g'_{\gamma'})$,
whence $L\trianglelefteq_2 \langle b, L\rangle \leq \mathfrak S_F$.

Since  the image  $\im(\mu)=:H$  of the monodromy map  has order $2|H^0|=|f^{-1}(x')|$, and $H^0\trianglelefteq_2 \im(\mu)$, the claim is proved.\qed

\begin{remark}\label{rem:lift}

An analogous argument   works   also in the following cases:
\begin{itemize}

\item the generating vector of $H^0$ is of type $[2;-]$. In this case we lift the hyperelliptic involution,
getting a group $H$ of order $2|H^0|$ and a covering $C\to  C/H\cong \mathbb P^1$ branched in 6 points (images of the 6 Weierstra\ss-points via the hyperelliptic map) with multiplicity 2,
i.e., having  generating vector of type $[0;2^6]$.

\item The generating vector of $H^0$ is of type $[0; m,m,m]$. 
In this case we lift the symmetric group $S_3$: the group of symmetries of $\mathbb P^1$ with marked points 
$1$, $\xi:=\exp\left(\frac{2\pi i}3\right)$ and $\xi^2$ (cf.~\cite[Section 2.1]{FL19}). In this way we get a  group $H$ of order $6|H^0|$ and a covering $C\to  C/H\cong \mathbb P^1$  with generating vector of type
$[0;2,3,2m]$.

\end{itemize}

\end{remark}

\section{Proof of the Main Theorem}\label{sec:proof_main}

In this section we prove the following theorem. 
Our main theorem is then a combination of this theorem and the results of \cite{CF18} mentioned in the  Introduction.

\begin{theorem}\label{mainThm}
Let $X:=(C\times C)/G$ be a semi-isogenous mixed surface of general type with  $p_g(X)=q(X)$, $K^2_X>0$ and 
$|G^0| \leq 2000, \neq 1024$, and let $X_{min}$ be its minimal model. Then the followings hold.
\begin{itemize}
\item If $K^2_X=2$, then there are two $(-1)$-curves on $X$, and $K_{X_{min}}^2=4$.
\item If $K^2_X=4$, then there is one  $(-1)$-curve on $X$, and $K_{X_{min}}^2=5$.
\end{itemize}
\end{theorem}

\begin{proof}

The claims have been already proved   in \cite[Section 7]{CF18} for $p_g(X)=q(X)=2$ (see also Remark \ref{rmk:OD2}).

According to Tables \ref{q0}  and \ref{q1},  there are 3 main  cases left, and we prove the statement case by case, considering a suitable automorphism group
$H$: $G^0<H<\Aut(C)$, and constructing the induced orbit divisors.
In each case we find the maximal possible number of $(-1)$-curves predicted by Proposition \ref{bound_curves}.

\smallskip

\textbf{Case 1, $p_g(X)=q(X)=0$.} 
Looking at Table \ref{q0}, we have to consider 2 families, 
both with $K^2_X=2$, and  determined by a group $G$ of order 256,  a
subgroup $G^0<G$ of order 128 and a generating vector of type $[0;2^5]$ for $G^0$.
As automorphism group $H$ we consider the group $G^0$.

The MAGMA script in Appendix \ref{script} shows that in both cases the 128 curves $\Gamma_f$, $f\in G^0$, induce 18 orbit divisors on 
$X$, and among them there are two $(-1)$-curves, which cannot meet, since $X$ is of general type.
 By  Proposition \ref{bound_curves}, there are no further $(-1)$-curves, whence contracting them we get the minimal model $X_{min}$ of $X$: $K_{X_{min}}^2=4$.

\smallskip

\textbf{Case 2, $p_g(X)=q(X)=1$, $K^2_X=4$.} 
Looking at Table \ref{q0}, we have to consider 2 families, both  determined by a group $G$ of order 48,  a subgroup $G^0<G$ of order 24 and a generating vector of type $[1;2,2]$ for $G^0$, i.e.,  the covering $C\to C/G^0$ is branched in 2 points and $C/G^0$ has genus 1.
By Lemma \ref{lemma:lifting}, there exists a degree 2 extension $H$ of $G^0$ acting as a group of automorphism of $C$.

The supporting MAGMA script available on the author web-page shows that in both cases
among the orbit divisors induced by the 48 curves $\Gamma_f$, $f\in H$ there  is a  $(-1)$-curve. 
 By  Proposition \ref{bound_curves}, there are no further $(-1)$-curves, whence contracting it we get the minimal model $X_{min}$ of $X$: $K_{X_{min}}^2=5$.

\smallskip

\textbf{Case 3, $p_g(X)=q(X)=1$, $K^2_X=2$.} 
This case is analogous to Case 2: all families are 
determined by a group $G$ of order 64,  a subgroup $G^0<G$ of order 32 and a generating vector of type $[1;2,2]$ for $G^0$, 
and there exists a degree 2 extension $H$ of $G^0$ acting as a group of automorphism of $C$ (Lemma \ref{lemma:lifting}).

The supporting MAGMA script available on the author web-page shows that in each case among the orbit divisors induced by the 64 curves $\Gamma_f$, $f\in H$ there  are two $(-1)$-curves, which cannot meet, since $X$ is of general type.
 By  Proposition \ref{bound_curves}, there are no further $(-1)$-curves, whence contracting them we get the minimal model $X_{min}$ of $X$: $K_{X_{min}}^2=4$.
\end{proof}

\begin{remark}\label{rmk:OD2}

 In \cite[Section 7.1]{CF18} we explicitly construct a $(-1)$-curve on each semi-isogenous mixed surface with $p_g(X)=q(X)=2$, $K^2_X=4$, using   the technique of the orbit-divisors in a rough way. 
It is indeed worth mentioning that the arguments  used in the  proof of Theorem \ref{mainThm} also work  in the case $p_g(X)=q(X)=2$ (see also Remark \ref{rem:lift}).


\end{remark}

\appendix

 \section{Magma script}\label{script}

The script can be freely run at \url{http://magma.maths.usyd.edu.au/calc/}.

\begin{code_magma}
/* This function counts the number of fixed points of the automorphism f in H < Aut(C) */
CountingIntersections:=function(f,seq,H)
int:=0;
for j in [1..#seq] do  	c:=0; K:=sub<H|seq[j]>;
	for g in H do 
		if f in  {g*k*g^-1: k in K}  then c:=c+1; 
	end if; end for; 
	int:=int+(c/#K); 
end for; return int;
end function;

/* The algebraic data defining the family are: */
G:=SmallGroup(256, 47930); genus:=33;
gv:=[G.2*G.5*G.6*G.7*G.8, G.1*G.3*G.7, G.2*G.3*G.4*G.5*G.8, G.1*G.2*G.5*G.6*G.8, G.2*G.4*G.5*G.6];

/* For the other family  replace with:
G:=SmallGroup(256, 45303); genus:=33; 
gv:=[G.1*G.2*G.3*G.5*G.6*G.7*G.8, G.3*G.5*G.6, G.3*G.6, G.2*G.4*G.5*G.6*G.7, G.1*G.3*G.4*G.5];
*/

G0:=sub<G|gv>; t:=Rep({x: x in G | x notin G0}); /* t= \tau' */
O2:={x: x in G| x notin G0 and Order(x) eq 2};
R:={t*x: x in O2};   /* the ramification curves represented as graph of  elements in G0 */

/* We construct now the orbit divisors, distinguishing between branch and non-branch curves */
Curves:=[]; P:=Set(G0);
P:=P diff R;/* We discard the branch curves, since we know they are not rational

while not IsEmpty(P) do 	f:=Rep(P); /* it represents the curve Delta_f, graph of the automorphism f of C  */
	Orbitf:={t*h*t^-1*f*h^-1 : h in G0} join {t^2*h*f^-1*t*h^-1*t^-1 : h in G0};   /* it is the G-orbit  of Delta_f */
	Append(~Curves, Orbitf); 	P:=P diff Orbitf;
end while; /* Each element in Curves represents an orbit divisor, which is not a branch curve */

/* For each orbit divisor D we compute  D^2 and K_X.D, and we save in Curves_1 the (-1)-curves */
Curves_1:={};
for a in [1..#Curves] do  	Gamma:=Curves[a]; n:=#Gamma; 	self:= -2*n*(genus-1); 	GK:=4*n*(genus-1); 
	for g1 in Gamma do
		for g2 in Gamma diff {g1} do 
			int:=CountingIntersections(g1^-1*g2,gv,G0); self:=self+int;
		end for; 
		for r in R  do 
			intK:=CountingIntersections(g1^-1*r,gv,G0); GK:=GK-intK; 
	end for; end for;
	if self/#G eq -1 and GK/#G eq -1 then  Include(~Curves_1, a);	end if;
end for;

 #Curves_1; /* it returns the number of (-1)-curves  */
\end{code_magma}

\section{Semi-isogenous mixed surfaces with $p_g=q$ and $K^2>0$}
\label{app_tables}

For the benefit of the reader, we collect here 
the semi-isogenous mixed surfaces of general type with $p_g=q$ and $K^2>0$  classified in \cite{CF18}.

In Table  \ref{q2}, \ref{q0}, \ref{q1}, every row corresponds to a family and 
 we use the following notation: 
columns  $Id(G)$ and $Id(G^0)$  report the MAGMA identifier of the groups $G$ and $G^0$:
 the pair $(a,b)$ denotes the $b^{th}$ group of order $a$ in the database of Small Groups.
The groups $(256,3678)$ and $(256,3679)$  do not have a representation as semidirect product of non trivial 
groups of smaller order, so we leave the relative spots blank.

The column Type gives the type of the generating vector for  $G^0$  is a short way, e.g.~$[0; 2^5]$ stands for
$[0; 2,2,2,2,2]$.
The Branch Locus $\mathcal B$ of $\eta\colon C\times C \to X$ is also given in a short way, e.g.~$(3,-8)^2, (2,-4)^2$ means 
that $\mathcal B$ consists of 4 curves, two of genus 2 and  self-intersection $-4$ and two 
of genus 3 and  self-intersection $-8$.

Finally, we label by (\#)  those surfaces, whose minimal model  is determined in the present paper.

\begin{table}[!ht]
\caption{$p_g=q=2$, $K^2>0$}\label{q2}
\begin{tabular}{c|c|c|c|c|c|c|c|cc}
$K_X^2$&$G$ & $Id(G)$ & $G^0$ & $Id(G^0)$ & $g(C)$ & Type & Branch Locus $\mathcal B$ &  $H_1(X,\mathbb Z)$&   \\
\hline
8&  $\mathbb{Z}_4$ &4,1&$\mathbb{Z}_2$&2,1& 3&[2;-] &$\emptyset$&$\mathbb Z_2\times\mathbb Z^4$& \\
\hline
7& $ \mathbb{Z}_6$&6,2& $ \mathbb{Z}_3$&3,1& 4&[2;-] &$(2,-4)$&$\mathbb Z^4$&  \\
\hline
6& $D_4$&8,3&  $\mathbb{Z}_2^2$ &4,2& 5 &[2;-]& $(3,-8)$& $\mathbb Z^4$& \\
6&  $D_4$&8,3& $\mathbb{Z}_2^2$&4,2& 5 &[2;-]& $(3,-8)$&$\mathbb Z_2\times\mathbb Z^4$& \\
6&  $\mathbb{Z}_2\times\mathbb{Z}_4$&8,2& $\mathbb{Z}_4$& 4,1&5&[2;-]& $(2,-4)^2$& $\mathbb Z^4$&\\
\hline
4&$ D_6 $&12,4&$S_3$&6,1&7& [2;-]&$(2,-4),(4,-12)$ & $\mathbb Z^4=\pi_1(X)$&  \\
\hline
2 & $\mathbb{Z}_2\times D_4$&16,11& $ D_4$&8,3&9&[2;-]&$(2,-4)^2,(3,-8)^2$& $\mathbb Z^4=\pi_1(X)$&  \\
2 & $\mathbb{Z}_2\times D_4$&16,11& $ D_4$&8,3&9&[2;-]&$(2,-4)^2,(3,-8)^2$& $\mathbb Z^4=\pi_1(X)$&  \\
2 & $ D_4\rtimes\mathbb{Z}_2$ &16,13& $Q$&8,4&9&[2;-]&$(3,-8)^3$& $\mathbb Z^4=\pi_1(X)$&  \\
\end{tabular}
\end{table}

\begin{table}[!ht]
\caption{$p_g=q=0$, $K^2>0$}\label{q0}

\begin{tabular}{c|c|c|c|c|c|c|c|cc}
$K_X^2$&$G$ & $Id(G)$ & $G^0$ & $Id(G^0)$ & $g(C)$ & Type & Branch Locus $\mathcal B$ &  $H_1(X,\mathbb Z)$   \\
\hline\hline
8 & $D_{2,8,5}\rtimes\mathbb{Z}_2^2$ & 64, 92& $\mathbb{Z}_2^2\times D_4$ & 32, 46&9&[0;$2^5$]& $\emptyset$ &$ \mathbb Z_2^3\times\mathbb Z_8$ \\
8 & & 256, 3679& $(\mathbb{Z}_2^3\rtimes \mathbb{Z}_4)\rtimes \mathbb{Z}_4$ & 128, 36&17&[0;$4^3$]& $\emptyset$ &$\mathbb Z_2^2\times\mathbb Z_4^2$ \\
8 & & 256, 3678& $(\mathbb{Z}_2^3\rtimes \mathbb{Z}_4)\rtimes \mathbb{Z}_4$ & 128, 36&17&[0;$4^3$]& $\emptyset$ &$\mathbb Z_2^2\times\mathbb Z_4^2$ \\
8 & & 256, 3678& $(\mathbb{Z}_2^3\rtimes \mathbb{Z}_4)\rtimes \mathbb{Z}_4$ & 128, 36&17&[0;$4^3$]& $\emptyset$ &$ \mathbb Z_2^4\times\mathbb Z_4$ \\
8 & & 256, 3678& $(\mathbb{Z}_2^3\rtimes \mathbb{Z}_4)\rtimes \mathbb{Z}_4$ & 128, 36&17&[0;$4^3$]& $\emptyset$ &$ \mathbb Z_4^3$ \\
\hline
6& $\mathbb{Z}_8\rtimes\mathbb{Z}_2^2$& 32,43& $\mathbb{Z}_2\times D_4$ & 16,11& 9 & [0; $2^6$]& $(3,-8)$&$\mathbb Z_2^2\times\mathbb Z_4^2$ \\

6& $\mathbb{Z}_2^4\rtimes\mathbb{Z}_2$ & 32,27& $\mathbb{Z}_2^4$ & 16,14& 9 & [0; $2^6$]& $(3,-8)$&$\mathbb Z_2^2\times\mathbb Z_4^2$ \\

6& $\mathbb{Z}_2^4\rtimes\mathbb{Z}_2$ & 32,27& $\mathbb{Z}_2^4$ & 16,14& 9 & [0; $2^6$]& $(3,-8)$&$\mathbb Z_4^3$ \\

6& $\mathbb{Z}_7\times D_7$ & 98,3& $\mathbb{Z}_7^2$ & 49,2& 15 & [0; $7^3$]& $(3,-8)$&$\mathbb Z_7^2$ \\

6& $\mathbb{Z}_7\times D_7$ & 98,3& $\mathbb{Z}_7^2$ & 49,2& 15 & [0; $7^3$]& $(3,-8)$&$\mathbb Z_7^2$ \\

6& $\mathbb{Z}_4^2\rtimes D_4$ & 128, 734& $\mathbb{Z}_4^2\rtimes\mathbb{Z}_2^2$ &64,211& 17& [0; $2^5$]& $(3,-8)$&$\mathbb Z_2\times\mathbb Z_4^2$ \\

6& $(\mathbb{Z}_2^2\times D_8)\rtimes\mathbb{Z}_2$ & 128, 750& $\mathbb{Z}_2^2\times D_8$ & 64,250& 17 & [0; $2^5$]& $(3,-8)$&$\mathbb Z_2\times\mathbb Z_4^2$ \\

6& $(\mathbb{Z}_2\times D_8)\rtimes\mathbb{Z}_2^2$ & 128, 1797& $\mathbb{Z}_2^2\times D_8$ & 64,250& 17 & [0; $2^5$]& $(2,-4)^2$&$\mathbb Z_2\times\mathbb Z_4^2$ \\
\hline
2&$(\mathbb{Z}_2^3\rtimes D_4)\rtimes \mathbb{Z}_2^2$ & 256, 47930&$\mathbb{Z}_2^4\rtimes D_4$ & 128, 1135& 33 & [0; $2^5$]& $(3,-8)^3$&$\mathbb Z_2^3\times\mathbb Z_4$&(\#)\\
2&$(\mathbb{Z}_4^2\rtimes \mathbb{Z}_2^2)\rtimes \mathbb{Z}_2^2$ & 256, 45303& $\mathbb{Z}_2^4\rtimes D_4$ & 128, 1135& 33 & [0; $2^5$]& $(3,-8)^2, (2,-4)^2$&$\mathbb Z_2^3\times\mathbb Z_4$ &(\#)\\
\end{tabular}
\end{table}

\begin{table}[!ht]\caption{$p_g=q=1$, $K^2>0$}\label{q1}
\begin{tabular}{c|c|c|c|c|c|c|c|cc}
$K_X^2$ &$G$ & $Id(G)$ & $G^0$ & $Id(G^0)$ & $g(C)$ & Type & Branch Locus $\mathcal B$ &  $H_1(X,\mathbb Z)$&  \\
\hline\hline
8& $D_{2,8,5}$ & 16,6 & $\mathbb{Z}_2\times\mathbb{Z}_4$ & 8,2 &5& [1;$2^2$]& $\emptyset$& $\mathbb Z_4\times\mathbb Z^2$&  \\
8& $D_{2,8,3}$ & 16,8 & $D_4$ & 8,3 &5& [1;$2^2$]& $\emptyset$& $\mathbb Z_4\times\mathbb Z^2$&  \\
8& $\mathbb{Z}_2^2\rtimes\mathbb{Z}_4$&16,3 & $\mathbb{Z}_2^3$ & 8,5 &5& [1;$2^2$]& $\emptyset$& $\mathbb Z_2^3\times\mathbb Z^2$& \\
\hline
7& $\mathbb{Z}_3\rtimes D_4$ & 24,8& $D_6$ & 12,4& 7& [1;$2^2$]& $(2,-4)$& $\mathbb Z_2\times\mathbb Z^2$&  \\
7& $\mathbb{Z}_3\times D_4$ & 24,10& $\mathbb{Z}_2\times\mathbb{Z}_6$ & 12,5& 7& [1;$2^2$]& $(2,-4)$& $\mathbb Z_2\times\mathbb Z^2$&  \\
\hline
6& $D_4$ &8,3& $\mathbb{Z}_2^2$ & 4,2& 5 & [1;$2^4$]& $(3,-8)$&$\mathbb Z_2^3\times\mathbb Z^2$&   \\
6& $\mathbb{Z}_3\times S_3$ & 18,3& $\mathbb{Z}_3^2$ & 9,2& 7 & [1;$3^2$]& $(3,-8)$&$\mathbb Z_3\times\mathbb Z^2$&  \\
6& $\mathbb{Z}_8\rtimes\mathbb{Z}_2^2$ & 32,43& $\mathbb{Z}_2\times D_4$ & 16,11& 9 & [1;$2^2$]& $(3,-8)$&$\mathbb Z_2\times \mathbb Z_4\times\mathbb Z^2$&   \\
6& $\mathbb{Z}_2^2\rtimes D_4$ & 32,28& $\mathbb{Z}_2\times D_4$ & 16,11& 9 & [1;$2^2$]& $(2,-4)^2$&$\mathbb Z_2\times \mathbb Z_4\times\mathbb Z^2$&  \\
6& $\mathbb{Z}_2^2\rtimes D_4$ & 32,28& $\mathbb{Z}_2\times D_4$ & 16,11& 9 & [1;$2^2$]& $(3,-8)$&$\mathbb Z_2^3\times\mathbb Z^2$&  \\
6&$\mathbb{Z}_4^2\rtimes\mathbb{Z}_2$ & 32,11& $\mathbb{Z}_4^2$ & 16,2& 9 & [1;$2^2$]& $(3,-8)$&$\mathbb Z_2\times\mathbb Z^2$&  \\
6& $D_8\rtimes\mathbb{Z}_2$ & 32,42& $D_4\rtimes\mathbb{Z}_2$ & 16,13& 9 & [1;$2^2$]& $(3,-8)$&$\mathbb Z_2^2\times\mathbb Z^2$&\\
6& $\mathbb{Z}_4^2\rtimes\mathbb{Z}_2$ & 32,31& $\mathbb{Z}_2^2\rtimes\mathbb{Z}_4$ & 16,3& 9 & [1;$2^2$]& $(3,-8)$&$\mathbb Z_2^2\times\mathbb Z^2$  &  \\
6& $(\mathbb{Z}_2^2\times\mathbb{Z}_4)\rtimes\mathbb{Z}_2$ & 32,30& $\mathbb{Z}_2^2\rtimes\mathbb{Z}_4$ & 16,3& 9 & [1;$2^2$]& $(2,-4)^2$&$\mathbb Z_4\times\mathbb Z^2$ &  \\
6& $D_{2,8,5}\rtimes\mathbb{Z}_2$ & 32,38& $\mathbb{Z}_2\times\mathbb{Z}_8$ & 16,5& 9 & [1;$2^2$]& $(2,-4)^2$&$\mathbb Z_2\times\mathbb Z^2$  &  \\
6& $\mathbb{Z}_4\times D_4$ & 32,25& $\mathbb{Z}_2^2\times\mathbb{Z}_4$ & 16,10& 9 & [1;$2^2$]& $(2,-4)^2$&$\mathbb Z_2^2\times\mathbb Z^2$  &  \\
6& $(\mathbb{Z}_2^2\times\mathbb{Z}_4)\rtimes \mathbb{Z}_2$ & 32,30& $\mathbb{Z}_2^2\times\mathbb{Z}_4$& 16,10& 9 & [1;$2^2$]& $(3,-8)$&$\mathbb Z_2^2\times\mathbb Z^2$&   \\
\hline
4& $S_3\times D_4$ & 48,38& $\mathbb{Z}_2^2\times S_3$ & 24,14& 13& [1;$2^2$]& $(2,-4), (4,-12)$& $\mathbb Z_2^2\times\mathbb Z^2$&  (\#) \\
4& $D_{12}\rtimes\mathbb{Z}_2$ & 48,37& $\mathbb{Z}_4\times S_3$ & 24,5& 13& [1;$2^2$]& $(2,-4), (4,-12)$& $\mathbb Z_2\times\mathbb Z^2$&   (\#) \\
\hline

2& $(\mathbb{Z}_8\rtimes\mathbb{Z}_2^2)\rtimes\mathbb{Z}_2$ & 64,153& $D_{2,8,5}\rtimes\mathbb{Z}_2$ & 32,7& 17 & [1;$2^2$]& $(3,-8), (5,-16)$&$\mathbb Z_2\times\mathbb Z^2$&   (\#) \\
2& $\mathbb{Z}_8\rtimes D_4$ & 64,150& $D_4\rtimes\mathbb{Z}_4$ & 32,9& 17 & [1;$2^2$]& $(3,-8), (5,-16)$&$\mathbb Z_2\times\mathbb Z^2$&  (\#)\\
2& $\mathbb{Z}_2^2\rtimes D_8$ & 64,147& $D_4\rtimes\mathbb{Z}_4$ & 32,9& 17 & [1;$2^2$]& $(2,-4)^2, (5,-16)$&$\mathbb Z_2\times\mathbb Z^2$  &(\#)\\
2& $(\mathbb{Z}_2\times D_8)\rtimes\mathbb{Z}_2$ & 64,128& $\mathbb{Z}_2\times D_8$ & 32,39& 17 & [1;$2^2$]& $(2,-4)^2,(3,-8)^2$&$\mathbb Z_2\times\mathbb Z^2$&  (\#)\\
2& $Q\rtimes D_4$ & 64,130& $\mathbb{Z}_2\times D_{2,8,3}$ & 32,40& 17 & [1;$2^2$]& $(3,-8)^3$&$\mathbb Z_2\times\mathbb Z^2$&  (\#)\\
2& $D_4\rtimes D_4$ & 64,134& $\mathbb{Z}_8\rtimes\mathbb{Z}_2^2$ & 32,43& 17 & [1;$2^2$]& $(3,-8)^3$&$\mathbb Z_2\times\mathbb Z^2$&  (\#)\\

2& $(\mathbb{Z}_2\times D_4)\rtimes\mathbb{Z}_2^2$ & 64,227& $\mathbb{Z}_2^3\rtimes\mathbb{Z}_4$ & 32,22& 17 & [1;$2^2$]& $(3,-8)^2,(2,-4)^2$&$\mathbb Z_2^2\times\mathbb Z^2$&  (\#)\\
2& $(\mathbb{Z}_2\times D_4)\rtimes\mathbb{Z}_2^2$ &64,227& $\mathbb{Z}_2^3\rtimes\mathbb{Z}_4$& 32,22& 17 & [1;$2^2$]& $(3,-8)^2,(2,-4)^2$&$\mathbb Z_2^2\times\mathbb Z^2$&  (\#)\\
2& $\mathbb{Z}_4\rtimes(D_4\rtimes\mathbb{Z}_2)$ & 64,228& $(\mathbb{Z}_4\rtimes \mathbb{Z}_4)\times\mathbb{Z}_2$ & 32,23& 17 & [1;$2^2$]& $(3,-8)^2,(2,-4)^2$&$\mathbb Z_2^2\times\mathbb Z^2$&  (\#)\\
2& $(\mathbb{Z}_4\times D_4)\rtimes\mathbb{Z}_2$ & 64,234& $(\mathbb{Z}_4\rtimes \mathbb{Z}_4)\times\mathbb{Z}_2$ & 32,23& 17 & [1;$2^2$]& $(3,-8)^3$&$\mathbb Z_2^2\times\mathbb Z^2$& (\#)\\
2& $(\mathbb{Z}_4\times D_4)\rtimes\mathbb{Z}_2$ & 64,234& $\mathbb{Z}_4^2 \rtimes\mathbb{Z}_2$ & 32,24& 17 & [1;$2^2$]& $(3,-8)^2,(2,-4)^2$&$\mathbb Z_2^2\times\mathbb Z^2$&  (\#)\\
2& $(\mathbb{Z}_4\rtimes Q)\rtimes\mathbb{Z}_2$  & 64,236& $\mathbb{Z}_4^2\rtimes\mathbb{Z}_2$ &32,24& 17 & [1;$2^2$]& $(3,-8)^3$&$\mathbb Z_2^2\times\mathbb Z^2$&  (\#)\\
2& $\mathbb{Z}_4^2\rtimes\mathbb{Z}_2^2$ & 64,219& $\mathbb{Z}_4\times D_4$ &32,25& 17 & [1;$2^2$]& $(3,-8)^3$&$\mathbb Z_2^2\times\mathbb Z^2$&(\#)\\
2& $(\mathbb{Z}_2^2\rtimes D_4)\rtimes\mathbb{Z}_2$ & 64,221& $\mathbb{Z}_4\times D_4$ & 32,25& 17 & [1;$2^2$]& $(3,-8)^3$&$\mathbb Z_2^2\times\mathbb Z^2$& (\#)\\
2& $(\mathbb{Z}_2\times\mathbb{Z}_4)\rtimes D_4$ & 64,213& $\mathbb{Z}_4\times D_4$ &32,25& 17 & [1;$2^2$]& $(3,-8)^2,(2,-4)^2$&$\mathbb Z_2^2\times\mathbb Z^2$& (\#)\\
2& $\mathbb{Z}_4^2\rtimes\mathbb{Z}_2^2$ & 64,206& $\mathbb{Z}_4\times D_4$ &32,25& 17 & [1;$2^2$]& $(3,-8),(2,-4)^4$&$\mathbb Z_2^2\times\mathbb Z^2$& (\#)
\end{tabular}
\end{table}

\end{document}